\documentclass[10pt]{amsart}

\usepackage{amsmath}
\usepackage{amssymb}
\usepackage{bm}
\usepackage{graphicx}
\usepackage{psfrag}
\usepackage{color}
\usepackage{xcolor}
\usepackage{hyperref}
\hypersetup{colorlinks=true, linkcolor=blue, citecolor=magenta, urlcolor=wine}
\usepackage{url}
\usepackage{algpseudocode}
\usepackage{fancyhdr}
\usepackage{mathtools}
\usepackage{tikz-cd}
\usepackage{xy}
\input xy
\xyoption{all}
\usepackage{stmaryrd}
\usepackage{calrsfs}
\usepackage{listings}
\usepackage{multicol}

\lstdefinestyle{mystyle}{language=C++,
	backgroundcolor=\color{white},   
	commentstyle=\color{codegreen},
	keywordstyle=\color{purple},
	numberstyle=\tiny\color{gray},
	stringstyle=\color{codepurple},
	basicstyle=\ttfamily\footnotesize,
	morekeywords={,size, logical\_or, to, step\_size, function, in, foreach,}
	breakatwhitespace=false,         
	breaklines=true,                 
	captionpos=b,                    
	keepspaces=true,                 
	numbers=left,                    
	numbersep=5pt,                  
	showspaces=false,                
	showstringspaces=false,
	showtabs=false,                  
	tabsize=4
}
\newtheorem*{maintheorem*}{Main Theorem}
\newtheorem{theorem}{Theorem}[section]
\newtheorem*{theorem*}{Main Theorem}

\newtheorem{prop}[theorem]{Proposition}

\newtheorem{lemma}[theorem]{Lemma}
\newtheorem{cor}[theorem]{Corollary}

\theoremstyle{definition}
\newtheorem{definition}[theorem]{Definition}

\newtheorem{example}[theorem]{Example}
\numberwithin{equation}{section}

\newcommand{\ee}{\mathbb{E}}

\newcommand{\nn}{\mathbb{N}}
\newcommand{\pp}{\mathbb{P}}
\newcommand{\aaa}{\mathcal{A}}

\newcommand{\ccc}{\mathcal{C}}
\newcommand{\ppp}{\mathcal{P}}

\newcommand{\qq}{\mathbb{Q}}
\newcommand{\rr}{\mathbb{R}}
\newcommand{\zz}{\mathbb{Z}}

\newcommand{\uu}{\mathcal{U}}

\providecommand\ldb{\llbracket}
\providecommand\rdb{\rrbracket}

\newcommand{\Pfin}{\mathcal{P}_{\mathrm{fin}}}
\newcommand{\Pfino}{\mathcal{P}_{\mathrm{fin},0}}
\newcommand{\Pfinu}{\mathcal{P}_{\text{fin}, \uu}}



\begin{document}
	
\mbox{}
\title{On primality and atomicity of numerical power monoids}

\author{Anay Aggarwal}
\address{PRIMES USA\\MIT\\Cambridge, MA 02139}
\email{anayagga@pdx.edu}

\author{Felix Gotti}
\address{Department of Mathematics\\MIT\\Cambridge, MA 02139}
\email{fgotti@mit.edu}

\author{Susie Lu}
\address{PRIMES USA\\MIT\\Cambridge, MA 02139}
\email{SusieLu@ohs.stanford.edu}

\date{\today}
	
\begin{abstract}
	Given a commutative monoid $M$, the finitary power monoid of $M$, denoted by $\Pfin(M)$ is the monoid consisting of all nonempty finite subsets of $M$ under the so-called sumset operation. The submonoid of $\Pfin(M)$ whose elements are those sets containing at least one invertible element of $M$ is called the restricted finitary power monoid of $M$ and is denoted by $\Pfinu(M)$. 
	
	In the first part of this paper, we establish a variation of a recent result by Bienvenu and Geroldinger on the (almost) non-existence of absolute irreducibles in (restricted) power monoids of numerical monoids: we argue the (almost) non-existence of primal elements in the same class of power monoids. The second part of this paper, devoted to the study of the atomic density of $\mathcal{P}_{\text{fin}, 0}(\mathbb{N}_0)$, is motivated by work of Shitov, a recent paper by Bienvenu and Geroldinger, and some questions pointed out by Geroldinger and Tringali. In the same, we study atomic density through the lens of the natural partition $\{ \mathcal{A}_{n,k} : k \in \mathbb{N}_0\}$ of $\mathcal{A}_n$, the set of atoms of $\mathcal{P}_{\text{fin}, 0}(\mathbb{N}_0)$ with maximum at most $n$: \[ \mathcal{A}_{n,k} = \{A \in \mathcal{A} : \max A \le n \text{ and } |A| = k\} \] for all $n,k \in \mathbb{N}$, where $\mathcal{A}$ is the set of atoms of $\mathcal{P}_{\text{fin}, 0}(\mathbb{N}_0)$. We pay special attention to the sequence $(\alpha_{n,k})_{n,k \ge 1}$, where $\alpha_{n,k}$ denote the size of the block $\mathcal{A}_{n,k}$. First, we establish some bounds and provide some asymptotic results for $(\alpha_{n,k})_{n,k \ge 1}$. Then, we take some probabilistic approach to argue that, for each $n \in \mathbb{N}$, the sequence $(\alpha_{n,k})_{k \ge 1}$ is almost unimodal. Finally, for each $n \in \mathbb{N}$, we consider the random variable $X_n : \mathcal{A}_n \to \mathbb{N}_0$ defined by the assignments $X_n : A \mapsto |A|$, whose probability mass function is $\mathbb{P}(X_n=k) = \alpha_{n,k}/| \mathcal{A}_n|$. We conclude proving that, for each $m \in \mathbb{N}$, the sequence of moments $(\mathbb{E}(X_n^m))_{n \ge 1}$ behaves asymptotically as that of a sequence $(\mathbb{E}(Y_n^m))_{n \ge 1}$, where $Y_n$ is a binomially distributed random variable with parameters $n$ and $\frac12$.

\end{abstract}

\maketitle

\bigskip
%%%%%%%%%%%%
%%%%%%%%%%%%
\section{Introduction}
\label{sec:intro}

Let $M$ be an (additively written) commutative semigroup with an identity element, called a monoid in the scope of this paper. The set $\Pfin(M)$ consisting of all nonempty finite subsets of $M$ is also a monoid under the sumset operation in~\eqref{eq:sumset}: for nonempty finite subsets $B$ and $C$ of $M$,
\begin{equation} \label{eq:sumset}
	B+C := \{b+c : (b,c) \in B \times C \}.
\end{equation}
The monoid $\Pfin(M)$ is called the (finitary) power monoid of $M$, while the submonoid $\mathcal{P}_{\text{fin}, \uu}(M)$ of $\Pfin(M)$ consisting of all finite nonempty subsets of $M$ containing at least one invertible element is called the restricted (finitary) power monoid of~$M$. Power monoids and restricted power monoids seem to be first investigated by Shafer and Tamura~\cite{TS67} back in the sixties. On the other hand, power monoids/semigroups have played a crucial role in the development of the automata theory and formal languages (see \cite{jA02,jP95} and references therein).
\smallskip

After the appearance of~\cite{FT18}, where Fan and Tringali established some connections between factorization theory and arithmetic combinatorics, a significant amount of interest in (restricted) power monoids seem to have resurfaced, especially from the viewpoint of atomicity and factorization theory~\cite{AT21,GLRRT24} and in connection with certain instances of the power-monoid isomorphism problem~\cite{sT24,TY23,TY24}, which is the problem of determining, given a class $\ccc$ of monoids, whether the (restricted) power monoid of each monoid $M$ in $\ccc$ uniquely determines $M$ (up to isomorphism). As a consequence, various authors have made recent contributions to the literature of (finitary) power monoids (see~\cite{TY23} and some of the references therein).
\smallskip

The present paper is our contribution to this recent trend, and it is dedicated to study the class consisting of (restricted) power monoids of numerical monoids. This class of monoids was first and recently investigated by Bienvenu and Geroldinger in~\cite{BG24}, which is, in many respects, the paper motivating several results we establish here. In their paper, Bienvenu and Geroldinger studied certain atomic and ideal-theoretical aspects of $\Pfin(N)$ and $\Pfino(N)$, which they call the numerical power monoid and the restricted numerical power monoid of $N$, where $N$ is a numerical monoid (the monoid $\Pfino(N)$ is simply $\mathcal{P}_{\text{fin}, \uu}(M)$, but the former notation is more convenient to emphasize that the only invertible element of a numerical monoid is~$0$). Also recently, atomic and factorization aspects of (finitary) power monoids of additive submonoids of $\qq_{\ge 0}$ (which generalize numerical monoids and account for all rank-one torsion-free monoids up to isomorphism) have been investigated by Gonzalez et al. in~\cite{GLRRT24}.
\smallskip

In Section~\ref{sec:background}, we briefly revise some fundamental notation and terminology we will use throughout this paper.
\smallskip

In Section~\ref{sec:primal elements}, we consider the notions of primal elements to establish a generalization of the non-existence of primes in (restricted) numerical power monoids (with the exception of $\Pfin(\nn_0)$, which contains the prime $\{1\}$). Our main result in this direction is motivated by~\cite[Theorem~4.11]{BG24}, which states that the only numerical power monoid that contains absolute irreducibles is $\Pfin(\nn_0)$ (whose only absolutely irreducible is the singleton~$\{1\}$), while no restricted numerical power monoid contains absolute irreducibles. We say that an irreducible element $a \in M$ (also called an atom) is an absolute irreducible if $na$ has a unique factorization for every $n \in \nn$. Also, we say that a non-invertible element $p \in M$ is a primal element if for any $b,c \in M$ the (additive) divisibility relation $p \mid_M b+c$ ensures the existence of $b',c' \in M$ with $p = b'+c'$ such that $b' \mid_M b$ and $c' \mid_M c$. The notion of a primal element was introduced by Cohn~\cite{pC68} back in the sixties in his study of Schreier domains. We establish here the following variation of~\cite[Theorem~4.11]{BG24}: restricted numerical power monoids do not contain primal elements, while the only numerical power monoid containing primal elements is $\Pfin(\nn_0)$, whose only primal element is~$\{1\}$.
\smallskip

Section~\ref{sec:a 2-dim graded partition of atoms} is devoted to the study of atomic density in the monoids $\Pfino(\nn_0)$. Our study is motivated by the work of Shitov~\cite{yS14} and the more recent work of Bienvenu anda Geroldinger~\cite[Section~6]{BG24}, where the authors carry out the first investigation of atomic density in (restricted) numerical power monoids. They consider, for each $n \in \nn$, the quotient $q_n$ between the number of atoms with maximum $n$ and the total number of elements with maximum $n$. For restricted numerical monoids, they found that $\lim_{n \to \infty} q_n = 1$ and, although this result was previously established by Shitov~\cite{yS14} in the language of Boolean polynomials, they improved upon the decay rate of the number of non-atoms. Here we will further partition each set
\[
	\aaa_n := \{A \in \mathcal{A}(\Pfino(\nn_0)) : \max A \le n \}
\]
into blocks by size, obtaining the set of blocks $\{ \aaa_{n,k} : k \in \nn \}$. Hence if $\aaa$ is the set of atoms of $\Pfino(\nn_0)$, then
\[
	 \aaa_{n,k} = \{A \in \aaa : \max A \le n \text{ and } |A| = k\}
\]
for every $k \in \nn$. This will allow us to look at the atomic density of $\Pfino(\nn_0)$ through the lens of the sequence $(\alpha_{n,k})_{n,k \ge 1}$, where $\alpha_{n,k} = |\aaa_{n,k}|$ for all $n,k \in \nn_0$. First, we establish some bounds and asymptotic results for $(\alpha_{n,k})_{n,k \ge 1}$. Then, we take some probabilistic approach to argue that, for each $n \in \nn$,  the sequence $(\alpha_{n,k})_{k \ge 1}$ is almost unimodal. Lastly, for each $n \in \nn$, we consider the random variable $X_n : \aaa_n \to \nn_0$ defined by the assignments $X_n \colon A \mapsto |A|$ whose probability mass function is
\[
	\pp(X_n=k) = \frac{\alpha_{n,k}}{|\aaa_n|}.
\]
We conclude this paper by proving that, for each $m \in \nn$, the sequence of moments $(\ee(X_n^m))_{n \ge 1}$ behaves asymptotically as the sequence $(\ee(Y_n^m))_{n \ge 1}$, where $Y_n$ is a binomially distributed random variable with parameters $n$ and $\frac12$.

\bigskip
%%%%%%%%%%%%
%%%%%%%%%%%%
\section{Background}
\label{sec:background}

\smallskip
%%%%%%%%%%%%%%%%
\subsection{General Notation}

As is customary, $\zz$, $\qq$, and $\rr$ will denote the set of integers, rational numbers, and real numbers, respectively. In addition, we let $\pp$, $\nn$, and $\nn_0$ denote the set of rational primes, positive integers, and nonnegative integers, respectively. For any $b,c \in \zz$,
\[
	\ldb b,c \rdb := \{n \in \zz : b \le n \le c\}.
\]
If $S \subseteq \rr$ and $r \in \rr$, then we set $S_{\ge r} := \{s \in S : s \ge r\}$ and $S_{> r} := \{s \in S : s > r\}$. Finally, for each nonzero $q \in \qq$, we let $\mathsf{n}(q)$ and $\mathsf{d}(q)$ denote, respectively, the unique $n \in \zz$ and $d \in \nn$ such that $q = \frac{n}d$ and $\gcd(n,d) = 1$.

\medskip
%%%%%%%%%%%%%%%%%%%
\subsection{Commutative Monoids}

Let $S$ be a commutative semigroup with identity. The group of invertible elements of $S$ is denoted by $\uu(S)$, and we say that $S$ is \emph{reduced} if $\uu(S)$ is the trivial group. To avoid cumbersome terminology, throughout this paper we reserve the term \emph{monoid}\footnote{In the standard literature, a monoid is defined to be a semigroup with an identity element.} for a cancellative, commutative, and reduced semigroup with an identity element. Furthermore, we tacitly assume that all monoids we mentioned here are written additively. Let $M$ be a monoid. We say that~$M$ is \emph{unit-cancellative} if the only pairs $(b,c) \in M \times M$ such that $b + c = b$ are those with $c = 0$. We obtain directly from the corresponding definitions that every cancellative monoid is unit-cancellative. For a subset $S$ of $M$, we let $\langle S \rangle$ denote the smallest submonoid of $M$ containing~$S$, and~$S$ is a \emph{generating set} of~$M$ provided that $M = \langle S \rangle$. We say that $M$ is \emph{finitely generated} if $M$ has a finite generating set.  
\smallskip

Additive submonoids of $\qq$ consisting of nonnegative rationals are referred to as \emph{Puiseux monoids} and they account, up to isomorphism, for all rank-one torsion-free monoids that are not groups (see \cite[Theorem~3.12.1]{GGT21}). In Section~\ref{sec:weaker notions of primes}, Puiseux monoids will be use as illustrative examples. The additive submonoids of $\nn_0$ account, up to isomorphism, for all finitely generated Puiseux monoids. A co-finite additive submonoid of~$\nn_0$ is called a \emph{numerical monoid}, and numerical monoids account up to isomorphism for all nontrivial submonoids of $\nn_0$. The \emph{Frobenius number} of a numerical monoid $N$ distinct from $\nn_0$ is $\max \, \nn_0 \setminus N$.
\smallskip

For $b,c \in M$, we say that $b$ (\emph{additively}) \emph{divides} $c$ in $M$ if $c = b + b'$ for some $b' \in M$, in which case, we write $b \mid_M c$. A submonoid $S$ of $M$ is called \emph{divisor-closed} provided that for all $b \in M$ and $c \in S$ the relation $b \mid_M c$ implies that $b \in S$. A non-invertible element $p \in M$ is called a \emph{prime} if for any $b,c \in M$ the relation $p \mid_M b+c$ implies either $p \mid_M b$ or $p \mid_M c$. An non-invertible element $p \in M$ is called \emph{primary} if for any $b,c \in M$ the the fact that $p \mid_M b+c$ ensures that either $p \mid_M b$ or $p \mid_M nc$ for some $n \in \nn$. It is clear that every prime element of $M$ is primary. Every nonzero element in a numerical monoid is primary, but the only numerical monoid containing primes is the free monoid $\nn_0$, whose only prime is $1$. Following Cohn~\cite{pC68}, we say that a non-invertible element $p \in M$ is \emph{primal}\footnote{In the definition of a primal element given by Cohn in~\cite{pC68}, he did not exclude invertible elements from the set of primal elements. However, we have done so here because we are thinking of primal elements as generalizations of prime elements.} if for any $b,c \in M$ the fact that $p \mid_M b+c$  ensures the existence of $b',c' \in M$ with $p = b'+c'$ such that $b' \mid_M b$ and $c' \mid_M c$. Clearly, if an element of $M$ is prime, then it must be primal. The only numerical monoid containing primal elements is $\nn_0$, where every nonzero element is primal.

\medskip
%%%%%%%%%%%%%%%
%%%%%%%%%%%%%%%
\subsection{Power Monoids}

As the notions of a primary/primal element given in the previous paragraph, the notion of an atom is a natural generalization of the notion of a prime. An non-invertible element $a \in M$ is called an \emph{atom} if whenever $a = u+v$ for some $u,v \in M$, then either $u$ or $ \in \uu(M)$ or $v \in \uu(M)$. Observe that an element is prime if and only if it is simultaneously an atom and a primal element. The set of atoms of~$M$ is denoted by $\mathcal{A}(M)$ and plays a crucial role in the scope of this paper. The monoid $M$ is \emph{atomic} if every non-invertible element can be written as a sum of atoms. Furthermore, an atomic monoid is called a \emph{unique factorization monoid} (UFM) if every non-invertible element has a unique factorization into atoms (up to order and associates). Inside a UFM, every nonzero element is primal. When $M$ is an atomic monoid, an atom $a \in M$ is called a \emph{strong atom} or an \emph{absolute irreducible}) provided that, for all $n \in \nn$, the only atom dividing $na$ is $a$ up to associate. Absolute irreducible elements are clearly primary.
\smallskip

Let $M$ be a monoid. Following \cite{FT18}, we let $\mathcal{P}_{\text{fin}}(M)$ denote the set consisting of all the nonempty finite subsets of~$M$, which is a monoid under \emph{sumset}, the following binary operation:
\[
	B+C := \{b+c : (b,c) \in B \times C \},
\]
where $B,C \in \mathcal{P}_{\text{fin}}(M)$. The monoid $\mathcal{P}_{\text{fin}}(M)$ is called the \emph{finitary power monoid} of $M$, although throughout this paper we will simply call $\mathcal{P}_{\text{fin}}(M)$ the \emph{power monoid} of $M$\footnote{It is worth emphasizing that the standard term `power monoid' is reserved in the literature for the set of all nonempty subsets of $M$ under sumset (without the finiteness restriction on the subsets).}. The submonoid
\[
	\mathcal{P}_{\text{fin}, 0}(M) := \{S \in \mathcal{P}_{\text{fin}}(M) : 0 \in S \}
\]
of $\Pfin(M)$ is called the \emph{restricted power monoid} of $M$. The (restricted) power monoid of any linearly orderable monoid is unit-cancellative \cite[Proposition~3.5(ii)]{FT18}.

When the monoid $M$ is reduced, the power monoid of any submonoid of $M$ is a divisor-closed submonoid of $\Pfino(M)$, as we proceed to verify.

\begin{prop} \label{prop:divisor-closed submonoids}
	Let $M$ be a reduced monoid, and let $N$ be a submonoid of $M$. Then the following statements hold.
	\begin{enumerate}
		\item $\Pfino(N)$ is a divisor-closed submonoid of $\Pfino(M)$.
		\smallskip
		
		\item $\mathcal{A}(\Pfino(N)) = \Pfino(N) \cap \mathcal{A}(\Pfino(M))$.
	\end{enumerate} 
\end{prop}

\begin{proof}
	(1) Fix $S \in \Pfino(N)$ and suppose that $S' \in \Pfino(M)$ divides $S$ in $\Pfino(M)$. Then take $T' \in \Pfino(M)$ such that $S' + T' = S$. As  $0 \in T'$, the inclusion $S' \subseteq S$ holds. Hence $S' \subseteq N$, and so $S' \in \Pfino(N)$. Thus, $\Pfino(N)$ is a divisor-closed submonoid of $\Pfino(M)$.
	\smallskip
	
	(2) Because the monoid $\Pfino(M)$ is reduced and contains $\Pfino(N)$ as a submonoid, the inclusion $\Pfino(N) \cap \mathcal{A}(\Pfino(M)) \subseteq \mathcal{A}(\Pfino(N))$ clearly holds. The reverse inclusion is an immediately consequence of the fact that $\Pfino(N)$ is a divisor-closed submonoid of $\Pfino(M)$.
\end{proof}

\bigskip
%%%%%%%%%%%%%%%%%%%%%%%%%%%%%%
%%%%%%%%%%%%%%%%%%%%%%%%%%%%%%
\section{A Variation of a Bienvenu-Geroldinger Theorem}
\label{sec:primal elements}

In this section, we begin our study of (restricted) power monoids of numerical monoids. The following terminology will turn out to be convenient.

\begin{definition}
	Let $N$ be a numerical monoid. We call $\mathcal{P}_{\text{fin}}(N)$ (resp., $\mathcal{P}_{\text{fin}, 0}(N)$) a \emph{numerical power monoid} (resp., a \emph{restricted numerical power monoid}).
\end{definition}

Since numerical monoids are linearly ordered monoid under the standard addition, (restricted) numerical power monoids are unit-cancellative. Also, as every numerical monoid is reduced, it follows from Proposition~\ref{prop:divisor-closed submonoids} that $\Pfino(N)$ is a divisor-closed submonoid of $\Pfino(\nn_0)$ for every numerical monoid~$N$: this allows us to reduce the proofs of many statements on the arithmetic/atomicity of $\Pfino(N)$ to $\Pfino(\nn_0)$. The submonoid of $\Pfin(\nn_0)$ consisting of all the singletons of $\nn_0$ is a divisor-closed submonoid of $\Pfin(\nn_0)$ (a characterization of the divisor-closed submonoids of $\Pfino(\nn_0)$ is given in~\cite[Theorem~4.3]{BG24}). In this direction, we can decompose $\Pfin(\nn_0)$ as follows \cite[Theorem~3.1]{BG24}:
\[
	\Pfin(\nn_0) = \big\{ \{n\} : n \in \nn_0 \big\} \oplus \Pfino(\nn_0).
\]

\medskip
%%%%%%%%%%%%%%%%%%%%%%%%%%%%%%
%%%%%%%%%%%%%%%%%%%%%%%%%%%%%%
\subsection{Primal Elements and Pre-schreier Monoids}

Recall that a non-invertible element $p \in M$ is primal if whenever $p \mid_M b+c$ for some $b,c \in M$, we can write $p = b'+c'$ for some $b', c' \in M$ such that $b' \mid_M b$ and $c' \mid_M c$. The monoid $M$ is called a \emph{pre-Schreier}\footnote{Schreier domains were introduced by Cohn in~\cite{pC68} as integrally closed domains whose nonunit elements are all primal elements.} monoid provided that every non-invertible element of $M$ is primal. Every UFM is a pre-Schreier monoid. Let us take a look at rank-one torsion-free monoids, which is one of the most natural classes generalizing that of numerical monoids.

\begin{example}
	Let $G$ be an additive subgroup of $\qq$ (i.e., a rank-one torsion-free abelian group), and set $M := G_{\ge 0}$. Observe that, for all $r,s \in M$, the conditions $r \le s$ and $r \mid_M s$ are equivalent. Fix a nonzero $p \in M$ and assume that  $p \mid_M b+c$ for some $b,c \in M$. Now write $p = b' + c'$, where $(b',c') := (b, p-b)$ if $p>b$ and $(b',c') := (p,0)$ if $p \le b$. In any case, we see that $b' \mid_M b$ and $c' \mid_M c$, whence $p$ is a primal element. Hence $M$ is a pre-Schreier monoid.
\end{example}

Numerical monoids, however, are almost never pre-Schreier monoids. The following example sheds some light upon this observation.

\begin{example}
	Let us argue that the only numerical monoid that is a pre-Schreier monoid is $\nn_0$. In order to do so, let $N$ be a numerical monoid with at least two atoms. Let $a_1$ and $a_2$ be two distinct atoms of $N$. Then take the minimum $m \in \nn$ such that $a_1 \mid_N m a_2$. Clearly, $m \ge 2$. After setting $b := a_2$ and $c := (m-1)a_2$, it is clear that $a_1 \mid_N b+c$ but as $a_1$ is an atom we cannot write $a_1 = b'+c'$ for some $b',c' \in N$ such that $b' \mid_N b$ and $c' \mid_N c$. Hence $N$ contains elements that are not primal and, therefore, it is not a pre-Schreier monoid.
\end{example}
%
%\begin{proof}
%	Now suppose that $M$ is a rank-one torsion-free monoid that is pre-Schreier, and suppose that $M$ is not an abelian group. It light of~\cite{GGT21}, we can assume that $M$ is a submonoid of the additive group $\qq_{\ge 0}$. We split the rest of the proof into the following cases.
%	\smallskip
%	
%	\textsc{Case 1:} $M$ contains a prime. Let $p$ be a prime of $M$. Observe that any nonzero $q \in M$, we can pick $c \in \nn$ such that $\mathsf{n}(p) \mid c q$ and, therefore, $p \mid_M c q$. After assuming that $c$ has been taking as smallest as it can possibly be, we obtain that $c=1$ (otherwise, $p \mid_M q + (c-1)q$ but neither $p \mid_M q$ nor $p \mid_M (c-1) q$). Thus, $q \in \nn_0 q$. Hence if $M$ contains a prime $p$, then $M = \nn_0 p$. Now suppose that $M$
%	\smallskip
%	
%	\textsc{Case 2:} $M$ does not contain primes. In this case $M$ must be an antimatter monoid because an element is prime if and only if it is a primal atom. Suppose, towards a contradiction, that there exist nonzero elements $q,r \in M$ such that neither $q + M \subseteq r + M$ nor $r + M \subseteq q + M$. Assume that $q < r$. Since $r$ is not an atom, we can a nonzero $\frac1p + r = s$
%\end{proof}

\medskip
%%%%%%%%%%%%%%%%%%%%%%%%%%%%%%%%
%%%%%%%%%%%%%%%%%%%%%%%%%%%%%%%%
\subsection{Primal Elements in Numerical Power Monoids}

In order to prove the main theorem of this section, we need the following lemma.

\begin{lemma} \label{lem:divisibility by a simple atom} \label{lem:lem_ext}
	Let $M$ be a monoid. For $S \in \ppp_{\emph{fin},0}(M)$ and $a \in M$, the following conditions are equivalent.
	\begin{enumerate}
		\item[(a)] $\{ 0, a \}$ divides $S$ in $\ppp_{\emph{fin},0}(M)$.
		\smallskip
		
		\item[(b)] $\ldb 0, \ell \rdb a$ divides $S$ in $\Pfino(M)$ for some $\ell \in \nn$.
		\smallskip
		
		\item[(c)] For each $x \in S$, either $x-a \in S$ or $x+a \in S$.
	\end{enumerate} 
\end{lemma}

\begin{proof}
	%For simplicity, set $\ppp(M) := \Pfino(M)$. 
	If $a = 0$, then the three conditions hold trivially. Therefore we assume that $a \neq 0$.
	\smallskip
	
	(a) $\Rightarrow$ (b): This is clear.
	\smallskip
	
	(b) $\Rightarrow$ (c): Let us assume now that $\ldb 0, \ell \rdb a$ divides $S$ in $\Pfino(M)$ for some $\ell \in \nn$. Set $A := \ldb 0, \ell \rdb a$ and write $S = S' + A$ for some $S' \in \Pfino(M)$. Fix $x \in S$. Because $S = S' + A$, we can take $k \in \ldb 0, \ell \rdb$ such that $x \in S' + k a$. If $k \ge 1$, then  $x-a = (x - ka) + (k-1)a \in S' + A = S$. Otherwise, $k=0$ and so $x \in S'$, from which we obtain that $x + a \in S' + A = S$. Hence either $x-a \in S$ or $x+a \in S$.
	\smallskip
	
	(c) $\Rightarrow$ (a): Finally, assume that for each $x \in S$ either $x-a \in S$ or $x+a \in S$. We will prove that the equality $S = S'  + \{0,a\}$ holds provided that
	\[
		S' := \{x \in M : x, x + a \in S\}.
	\]
	To argue the inclusion $S \subseteq S' + \{0,a\}$, fix $x \in S$. By our assumption, either $x-a \in S$ or $x+a \in S$. If $x-a \in S$, then the fact that $x \in S$ guarantees that $x-a \in S'$, whence $x \in S' + a \subseteq S' + \{0,a\}$. Otherwise, $x+a \in S$, and this time the fact that $x \in S$ guarantees that $x \in S'$, which clearly implies that $x \in S' + \{0,a\}$. Hence $S \subseteq S' + \{0,a\}$. Finally, note that the inclusions $S' + 0 \subseteq S$ and $S' + a \subseteq S$ are direct consequences of the definition of $S'$, and so $S' + \{0,a\} \subseteq S$, which is the other desired inclusion.
\end{proof}

We are in a position to prove the main theorem of this section: a variation of~\cite[Theorem~4.11]{BG24} for primal elements.

\begin{theorem} \label{prop:primal elements in NPMs}
	For a numerical monoid $N$, the following statements hold.
	\begin{enumerate}
		\item $\ppp_{\emph{fin},0}(N)$ does not contain primal element.
		\smallskip
		
		\item $\ppp_{\emph{fin}}(N)$ contains only one primal element, $\{1\}$.
	\end{enumerate} 
\end{theorem}

\begin{proof}
	Let $N$ be a numerical monoid.
	\smallskip
	
	(1) Set $\ppp := \Pfino(N)$. In order to argue this part, we first assume that $N = \nn_0$. Suppose, for the sake of a contradiction, that $\ppp$ contains primal elements. Let $P$ be a primal element of $\ppp$. As $P$ is a primal, it cannot be an atom as, otherwise, it would be a prime, but $\ppp$ does not contain any primes. As $\{0,1\}$ is an atom of $\ppp$, we obtain that $P \neq \{0,1\}$. Now observe that
	\[
		P + \ldb 0, \max P \rdb = \ldb 0,2 \max P \rdb = \{0,\max P \} + \ldb 0, \max P \rdb.
	\]
	Since $P$ is primal, we can write $P = B + C$ for some $B,C \in \ppp$ such that $B \mid_\ppp \{0, \max P \}$ and $C \mid_\ppp \ldb 0, \max P \rdb$. Since the only divisors of $\{0,\max P\}$ in $\ppp$ are $\{0\}$ and $\{0, \max P\}$, either $B = \{0\}$ or $B = \{0, \max P\}$. We split the rest of our argument into two cases, obtaining a contradiction in each of them.
	\smallskip
	
	\textsc{Case 1:} $B = \{ 0, \max P\}$. This implies that $P = \{ 0, \max P \}$. Because $P \neq \{0,1\}$, we obtain that $\max P > 1$. Therefore
	\[
		P + \ldb 0, \max P \rdb = \ldb 0, 2 \max P \rdb = \{0,1\} + \ldb 0, 2 \max P -1 \rdb.
	\]
	Thus, from the fact that $\{0,1\} \nmid_\ppp P$, we obtain that $P \mid_\ppp \ldb 0, 2 \max P - 1 \rdb$. As a consequence, $P \mid_\ppp \{0,1\} + \ldb 0,2 \max P - 2 \rdb$, and so $P \mid_\ppp \ldb 0,2 \max P - 2 \rdb$. By descent, we eventually reach a contradiction. %Thus, $\max P = 1$, and so $P = \{0,1\}$. However, $\{0,1\}$ is an atom. 
	\smallskip

	\textsc{Case 2:} $B= \{0\}$. This implies that $P = \ldb 0, \max P \rdb$. Suppose $\max P > 1$, and note that
	\[
		P + \{0, \max P - 1\} = \ldb 0, 2 \max P - 1 \rdb = \{0, \max P\} + \ldb 0, \max P - 1 \rdb.
	\]
	Because $P \nmid_\ppp \ldb 0, \max P - 1 \rdb$, it follows that $\{0, \max P\} \mid_\ppp P$, which in turn implies $P = \{ 0,\max P \}$, a contradiction.
	\smallskip
	
	Thus, we have proved that $\ppp$ contains no primal element when $N = \nn_0$. Now suppose that $N$ is a numerical monoid different from $\nn_0$ and, as we did before, take $P$ be a primal element of $\ppp$. Now fix $m \in \nn$ large enough so that $\zz_{\ge m} \subseteq N$ and $m > \max P$. Let us argue the following claim.
	\smallskip
	
	\noindent \textsc{Claim.} $P$ is a discrete interval.
	\smallskip
	
	\noindent \textsc{Proof of Claim.} Assume, towards a contradiction, that $P$ is not a discrete interval. Then, when sorted, $P$ contains a gap of size $b$, where $b \ge 2$. Observe that
	\[
		\{0, b-1\} + (\ldb m,2m \rdb \setminus \{m+b\}) = \ldb m,2m + b - 1 \rdb = P + \ldb m, 2m + b - 1 - \max P \rdb.
	\]
	Thus, $\{0, b - 1\}$ shares a divisor with $P$. This shared divisor cannot be $\{0, b-1\}$ because this means that $P$ has no gap of size $b$. Therefore $P$ must divide $\ldb m, 2m \rdb \setminus \{m+b\}$. However, this is not possible because any multiple $S$ of $P$ satisfies condition (c) in Lemma~\ref{lem:divisibility by a simple atom}. Thus, the claim is established.
	\smallskip
	
	Since $P$ is a discrete interval, $1 \in P$, which implies that $1 \in N$. However, this contradicts that $N = \nn_0$. Hence $\Pfino(N)$ does not contain primal elements.
	\smallskip
	
	(2) Now set $\ppp := \ppp_{\text{fin}}(N)$. As in the previous case, let us first assume that $N = \nn_0$. To prove that the only primal element of $\ppp$ is $\{1\}$, fix a primal element $P$ of $\ppp$. Since
	\[
		P + \ldb 0, \max P \rdb = \ldb \min P, 2 \max P \rdb = \{0, \max P \} + \ldb \min P, \max P\rdb,
	\]
	we see that $P$ shares a divisor with $\{0, \max P \}$, which implies that $0 \in P$. Then it follows from the previous part that $P = \{1\}$, which we know that is prime and, therefore, primal. Hence the only primal element of $\ppp_{\text{fin}}(\nn_0)$ is $\{1\}$.
	
	Finally, assume that $N$ is a numerical monoid different from $\nn_0$, and let $P$ be a primal element of~$\ppp$. We assume that $0 \notin P$. Take $m \in \nn$ large enough so that $\zz_{\ge m} \subseteq N$. Then
	\[
		P + \ldb 2m - \min P, 4m - \max P \rdb = \ldb 2m, 4m \rdb = \{0, m\} + \ldb 2m,3m \rdb.
	\]
	Therefore we obtain that $P$ shares a divisor with $\{0,m\}$. However, this implies $0 \in P$, a contradiction. Hence we can conclude that the only primal element of $\Pfin(N)$ is $\{1\}$.
\end{proof}

\bigskip
%%%%%%%%%%%%%%%%%%%%%%
%%%%%%%%%%%%%%%%%%%%%%
\section{A Partition on the Set of Atoms}
\label{sec:a 2-dim graded partition of atoms}

The study of atomic density in the setting of (restricted) numerical power monoids was also initiated by Bienvenu and Geroldinger in~\cite[Section~6]{BG24}. They study, given a numerical monoid $N$, the density of the sets of atoms of $\Pfino(N)$ and $\Pfin(N)$ by considering, for each $n \in \nn$, the quotient $q_n$ between the number of atoms with maximum $n$ and the total number of elements with maximum $n$. For restricted numerical monoids, they found that $\lim_{n \to \infty} q_n = 1$ and, although this result was previously established by Shitov~\cite{yS14} in the language of Boolean polynomials, they improved upon the decay rate of the number of non-atoms. 

Motivated by both the work of Shitov and the recent work of Bienvenu and Geroldinger, we proceed to further investigate the atomic density of the prototypical restricted numerical monoid $\Pfino(\nn_0)$. While we do not consider the whole class of (restricted) numerical power monoids, we propose a more sensitive method to investigate the atomic density of $\Pfino(\nn_0)$, which was kindly suggested by Geroldinger and Tringali as part of a list of related problems and questions to be studied in CrowdMath 2023. For each $n \in \nn$, we consider the sets
\[	
	\ppp_n := \{S \in \Pfino(\nn_0) : \max S \le n\} \quad \text{and} \quad \aaa_n := \{A \in \mathcal{A}(\Pfino(\nn_0)) : \max A \le n\},
\]
and then we partition $\aaa_n$ into blocks by size, introducing the sequence $(\alpha_{n,k})_{n,k \ge 1}$, whose $(n,k)$-th term is the number of atoms of $\Pfino(\nn_0)$ of size $k$ with maximum at most $n$.

%%%%%%%%%%%%%%%%%%%%%%%%%%%%%%
%%%%%%%%%%%%%%%%%%%%%%%%%%%%%%
\subsection{The Sequence $(\alpha_{n,k})_{n,k \ge 1}$}

For each $n \in \nn$, we aim to partition the set $\aaa_n$ by the sizes of its atoms: so we define the set
\begin{equation} \label{eq:A_{n,k}}
	\aaa_{n,k} := \big\{ A \in \mathcal{A}(\Pfino(\nn_0)) : \max A \le n \text{ and } |A| = k \big\}
\end{equation}
for every $k \in \nn_0$. Clearly, $\mathcal{A}_{n,0}$ is empty, and we define it only for convenience of notation. The sequence whose terms are the sizes of the sets $\aaa_{n,k}$ will be our central object of study for the rest of this section: for each $n \in \nn$, set
\begin{equation} \label{eq:alpha_{n,k}}
	\alpha_{n,k} := |\aaa_{n,k}| \quad \text{and} \quad \alpha_n := |\aaa_n|.
\end{equation}
for every $k \in \nn_0$. As a special case of~\cite[Theorem~6.1.1]{BG24}, we find that $\lim_{n \to \infty} \frac{|\aaa_n|}{|\ppp_n|} = 1$.  As a consequence, we obtain the following asymptotic formula (previously established by Shitov in~\cite{yS14}):
\begin{equation}\label{eq:Shitov's asymptotic for alpha(n)}
	\alpha_n = 2^n(1 - o(1)).
\end{equation}

Observe that $\alpha_{n,1} = 1$ as the only singleton in $\Pfino(\nn_0)$ is the identity element. We can also observe that $\alpha_{n,2} = n$ because $\{0,a\}$ is an atom of $\Pfino(\nn_0)$ for every $a \in \ldb 1,n \rdb$. Before finding an explicit formula for $\alpha_{n,3}$ and $\alpha_{n,4}$, we need the following lemma.
\smallskip

\begin{lemma} \label{lem:size of the addition}
	For nonzero $A, B \in \mathcal{P}_{\emph{fin}, 0}(\mathbb{N}_0)$, the following statements hold.
	\begin{enumerate}
		\item $|A+B| \ge |A| + |B| - 1$.
		\smallskip
		
		\item In addition, the following conditions are equivalent.
		\begin{enumerate}
			\item $|A+B| = |A| + |B| - 1$.
			\smallskip
			
			\item There exists $c \in \mathbb{N}$ such that $A = c \ldb 0, |A|-1 \rdb$ and $B = c \ldb 0, |B|-1 \rdb$, in which case, $A+B = c\ldb 0, |A|+|B|-2 \rdb$.
		\end{enumerate}
	\end{enumerate}
\end{lemma}

\begin{proof}
	(1) Let $A$ and $B$ be nonzero elements of $\mathcal{P}_{\text{fin}, 0}(\mathbb{N}_0)$, and set $n := |A|$ and $m := |B|$. Now let $a_1, \dots, a_{n-1}$ and $b_1, \dots, b_{m-1}$ be the positive integers with $a_1 < a_2 < \dots < a_{n-1}$ and $b_1 < b_2 < \dots < b_{m-1}$ such that $A := \{0, a_1, a_2, \dots, a_{n-1}\}$ and $B := \{0, b_1, b_2, \dots, b_{m-1}\}$. Now observe that
	\[
	S := \{0\} \cup \{a_j, a_{n-1} + b_k : j \in \ldb 1,n-1 \rdb \text{ and } k \in \ldb 1, m-1 \rdb \}
	\]
	is a subset of $A+B$. Since $0 < a_1 < a_2 < \dots < a_{n-1} < a_{n-1} + b_1 < a_{n-1} + b_2 < \dots < a_{n-1} + b_{m-1}$, it follows that $|S| = n+m-1$, which implies that $|A+B| \ge n + m -1 = |A| + |B| - 1$.
	\smallskip
	
	(2) To argue that (a) implies (b), assume that $|A+B| = |A| + |B| - 1$. Let $A$, $B$, and $S$ be as in the previous part, and set $a_0 = b_0 = 0$. We claim that $A = b_1 \ldb 0, |A| - 1 \rdb$. If $|A| = 2$, then our claim follows immediately. Thus, we assume that $|A| \ge 3$. Because $S \subseteq A + B$ and $|S| = |A| + |B| - 1 = |A+B|$, we see that $S = A+B$. Observe that $a_{n-2} + b_1$ is an element of~$S$ satisfying that $a_{n-2} < a_{n-2} + b_1 < a_{n-1} + b_1$, so the fact that $\ldb a_{n-2} + 1, a_{n-1} + b_1 - 1 \rdb \cap S = \{a_{n-1}\}$ guarantees that $a_{n-1} = a_{n-2} + b_1$. Now suppose that $a_k = a_{k-1} + b_1$ for some $k \in \ldb 2,n-1 \rdb$. Since $a_{k-2} + b_1 \in S$ and $a_{k-2} < a_{k-2} + b_1 < a_{k-1} + b_1 = a_k$, it follows from $\ldb a_{k-2}+1, a_k - 1 \rdb \cap S = \{a_{k-1}\}$ that $a_{k-1} = a_{k-2} + b_1$. Hence we conclude that $a_k = a_{k-1} + b_1$ for every $k \in \ldb 1,n-1\rdb$. This implies that $a_k = k b_1$ for each $k \in \ldb 0,n-1 \rdb$, and so $A = b_1\ldb 0, |A| - 1 \rdb$. In particular, we see that $a_1 = b_1$. In a similar manner, we can argue that $B = a_1 \ldb 0, |B| - 1 \rdb$, and then we just need to set $c = a_1$.
	
	Arguing that (b) implies (a) is immediate as if $A = c \ldb 0, |A|-1 \rdb$ and $B = c \ldb 0, |B|-1 \rdb$ for some $c \in \mathbb{N}$, the equality $A+B = c\ldb 0, |A|+|B|-2 \rdb$ clearly holds.
\end{proof}

We proceed to find explicit formulas for $\alpha_{n,3}$ and $\alpha_{n,4}$.

\begin{prop}
	For any $n \in \nn$, the following equalities hold:
	\begin{equation*}
		\alpha_{n,3} = \binom{n}2 - \bigg\lfloor \frac{n}2 \bigg\rfloor \quad \text{ and } \quad \alpha_{n,4} = \binom{n}{3}-\frac{1}{2}\binom{n}{2}+\frac{1}{2} \bigg\lfloor \frac{n}{2} \bigg\rfloor.
	\end{equation*}
\end{prop}

\begin{proof}
	To establish the formula for $\alpha_{n,3}$, we will count the number of non-atoms $A \in \ppp_n$ of $\Pfino(\nn_0)$ having size $3$. By definition, $A$ can be written as $B + C$, for some $B,C \in \Pfino(\nn_0)$ with $\min(|B|, |C|) \geq 2$. On the other hand, it follows from Lemma~\ref{lem:size of the addition} that $3 \ge |B| + |C| - 1$, so $|B| + |C| \leq 4$. Thus, $|B| = |C| = 2$. After writing $B = \{0, b\}$ and $C = \{0, c\}$ for some $b,c \in \nn$, we see that $A = \{0, b, c, b+c\}$. Since $|A| = 3$, the equality $b = c$ holds, so $A = \{0, b, 2b\}$. Hence $\left\lfloor \frac{n}{2} \right\rfloor$ non-atoms in $\Pfino(\nn_0)$ having size~$3$. Therefore
	\[
		\alpha_{n,3} = \binom{n}2 - \bigg\lfloor \frac{n}2 \bigg\rfloor.
	\]
	For $\alpha_{n,4}$, we will count the number of non-atoms $A$ in $\ppp_n$ of size $4$. By definition, $A$ can be written as $B + C$ for some $B,C \in \Pfino(\nn_0)$ such that $\min(|B|, |C|) \ge 2$. It follows from Lemma~\ref{lem:size of the addition} that $|B| + |C| \leq 5$. We can assume, without loss of generality, $|B| \leq |C|$. We consider the following two cases.
	\smallskip
	
	\textsc{Case 1:} $|B| = |C| = 2$. After writing $B = \{0, b\}$ and $C = \{0, c\}$ for some $b,c \in \nn$, we see that $A = \{0, b, c, b+c\}$. Since $A \subseteq \ldb 0,n \rdb$ and $|A| = 4$, it follows that $b \neq c$ and $b+c \leq n$. As a consequence, there are $\frac14( n^2-2n+(n \text{ mod } 2) )$ non-atoms $A$ in this case, where $n \text{ mod } 2$ denotes the remainder of $n$ upon division by 2.
	\smallskip
	
	\textsc{Case 2:} $|B| = 2$ and $|C| = 3$. Since $|B+C| = |B|+|C|-1$, Lemma~\ref{lem:size of the addition} implies that $B=\{0,c\}$ and $C=\{0,c,2c\}$ for some $c \in \nn$, so $A = \{0,c,2c,3c\}$. However, sets $A$ of this form have been already been counted in Case 1, so this case does not contribute any more non-atoms. 
	
	As a consequence, we obtain that
	\begin{equation*}
		\alpha_{n,4} = \binom{n}{3}-\frac{n^2-2n+(n \text{ mod } 2)}{4} = \binom{n}{3}-\frac{1}{2}\binom{n}{2}+\frac{1}{2} \bigg\lfloor \frac{n}{2} \bigg\rfloor.
	\end{equation*}
\end{proof}

\medskip
%%%%%%%%%%%%%%%%%%%%%%%%%%%%%%%%%%%
\subsection{Bounding the Sequence $(\alpha_{n,k})_{n,k \ge 1}$}

Our next goal is to establish (asymptotic) bounds for the sequence of sizes $(\alpha_{n,k})_{n,k \ge 1}$. We start by establishing an explicit upper bound.

\begin{theorem}
	For any $n,k \in \nn$ such that $k \in \ldb 1,n+1 \rdb$, the following inequality holds:
	\begin{equation}
		\alpha_{n,k} \le \binom{n}{k-1} - \binom{\lfloor n/2 \rfloor -1}{\lfloor k/2 \rfloor - 1}. 
	\end{equation}
\end{theorem}

\begin{proof}
	To prove this upper bound on $\alpha_{n,k}$, we will give a lower bound on the number of non-atoms $A \in \ppp_n$ of $\Pfino(\nn_0)$ having size $k$. 
	
	Say that $A \in \ppp_n$ of size $k$ is good if it can be written as $B+C$, where $B$ is $\{0,b\}$ for some positive integer $b \le n$ and $C \in \ppp_n$ has size at least $2$. It is clear that every good set $A$ is a non-atom. For each $i \in \ldb 0,b-1 \rdb$, let $S_i$ denote the set of integers in $\ldb 1,n \rdb$ that are $i \mod b$. Then a set $A$ is good if and only if $0 \in A$ and for each $i \in \ldb 0,b-1 \rdb$, when $S_i$ is written in increasing order, any element of $A$ that is in $S_i$ is adjacent to at least one other element of $A$ that is in $S_i$.
	
	We will prove a lower bound on the number of good sets for $b=\lfloor n/2 \rfloor$. First, the sets $S_i$ are given as follows: if $n$ is even, then $S_0 = \{0,\frac{n}{2},n \}$ and $S_i = \{i, i+\frac{n}{2} \}$ for each $i \in \ldb 1, \frac{n}{2} -1 \rdb$. If $n$ is odd, then $S_0 = \{0, \frac{n-1}{2}, n-1 \}$, $S_1 = \{1, \frac{n+1}{2}, n \}$, and $S_i = \{i, i+\frac{n-1}{2} \}$ for each $i \in \ldb 2, \frac{n-3}{2} \rdb$.
	
	We split our argument into two cases based on the parity of $k$. 
	\smallskip
	
	\textsc{Case 1:} If $k$ is even, then the number of such good sets is at least $\binom{\lfloor n/2 \rfloor -1}{(k-2)/2}$ because if we choose $\frac{k-2}{2}$ sets among $S_1, \dots, S_{\lfloor n/2 \rfloor-1}$, then the set $A$ consisting of all elements in the chosen $S_i$'s and the smallest two elements in $S_0$ is good. 
	\smallskip
	
	\textsc{Case 2:} If $k$ is odd, then the number of such good sets is at least $\binom{\lfloor n/2 \rfloor -1}{(k-3)/2}$  because if we choose $\frac{k-3}{2}$ sets among $S_1, \dots, S_{\lfloor n/2 \rfloor-1}$, then the set $A$ consisting of all elements in the chosen $S_i$'s and all three elements in $S_0$ is good.
	\smallskip
	
	Combining the two cases, we see that the number of good sets for $b=\lfloor n/2 \rfloor$ is at least $\binom{\lfloor n/2 \rfloor -1}{\lfloor k/2 \rfloor - 1}$, so the number of non-atoms in $\ppp_n$ of $\Pfino(\nn_0)$ having size $k$ is at least $\binom{\lfloor n/2 \rfloor -1}{\lfloor k/2 \rfloor - 1}$. Hence, $\alpha_{n,k}$ is at least $\binom{n}{k-1} - \binom{\lfloor n/2 \rfloor -1}{\lfloor k/2 \rfloor - 1}$.
\end{proof}

%After that we will take a probabilistic approach to obtain more general asymptotics for all $k=\epsilon n$ with $0<\epsilon<1$. Additionally, for any fixed $n \in \nn$, the finite sequence
%\[
%	(\alpha_{n,k})_{1 \le k \le n+1} := (\alpha_{n,1}, \alpha_{n,2}, \dots, \alpha_{n,n+1})
%\]
%seems to be unimodal. Using our asymptotics for $\alpha_{n,k}$, we proceed to establish a unimodality result of the same sequence $(\alpha_{n,k})_{1 \le k \le n+1}$: we prove that, for each $n \in \nn$, there are $n(1-o(1))$ values of $k$ such that the unimodality inequality holds for the sequence $(\alpha_{n,k})_{k \ge 1}$. 

\medskip
%%%%%%%%%%%%%%%%%%%%%%%%%%%%%%%%%%%%%%%%%%%
\subsection{A Strong Asymptotic on the Sequence $(\alpha_{n,k})_{n,k \ge 1}$}

Our next goal is obtaining a precise asymptotic on the sequence $(\alpha_{n,k})_{n,k \ge 1}$ for $0.111n<k<0.5n$, and we do so in this section by optimizing an argument given by Shitov in~\cite{yS14}. Before we begin, we need to establish an estimate of the binomial coefficients.

\begin{prop}\label{prop:asymptotics of binomial coefficients}
	If $k=\Omega(n)$, then
	\[
		\log \binom{n}{k}=(1+o(1))n\left(\frac{k}{n}\log \frac{n}{k}+\frac{n-k}{n}\log \frac{n}{n-k}\right).
	\]
	After letting $H(x) = -x\log x-(1-x) \log(1-x)$ be the entropy function, one obtains the following
	\[
		\log \binom{n}{k}=(1+o(1))n H(k/n).
	\]
\end{prop}

\begin{proof}
	It follows from Stirling's approximation that
	\[
		\binom{n}{k}=\frac{n!}{k!(n-k)!}=(1+o(1))\frac{\sqrt{2\pi n}(n/e)^n}{\sqrt{2\pi k}(k/e)^k \sqrt{2\pi(n-k)}((n-k)/e)^{n-k}}.
	\]
	Hence
	\[
		\binom{n}{k}=(1+o(1))\sqrt{\frac{n}{2\pi k(n-k)}}\left(\frac{n}{k}\right)^k\left(\frac{n}{n-k}\right)^{n-k}.
	\]
	Now we can take logarithms to obtain the following equality.
	\[
		\log \binom{n}{k}=O(1)+\log \sqrt{\frac{n}{2\pi k(n-k)}}+k\log \frac{n}{k}+(n-k)\log \frac{n}{n-k}.
	\]
	Because $k=\Omega(n)$, the second term in this sum can be swept under the rug. This implies the desired result.
\end{proof}

We need one more estimate regarding binomial coefficients before we can apply our results to the sequence $(\alpha_{n,k})_{n,k \ge 1}$.

\begin{prop}\label{prop:asymptotics of binomial coefficient ratio}
	Let $0<c<1/2$ and $t=\frac{\varepsilon n}{\log n}$ for some constant $\varepsilon>0$. Then
	$$\frac{\binom{n}{cn-t}}{\binom{n}{cn}}\le \exp\left(-\Omega\left(\frac{n}{\log n}\right)\right).$$
\end{prop}

\begin{proof}
	By Stirling's approximation,
	\[
		\frac{\binom{n}{cn-t}}{\binom{n}{cn}}=\frac{(cn)!(n-cn)!}{(cn-t)!(n-cn+t)!}=(1+o(1))\sqrt{\frac{cn(n-cn)}{(cn-t)(n-cn+t)}}\frac{(cn/e)^{cn}((n-cn)/e)^{n-cn}}{((cn-t)/e)^{cn-t}((n-cn+t)/e)^{n-cn+t}}.
	\]
	Hence
	\[
		\frac{\binom{n}{cn-t}}{\binom{n}{cn}}=(1+o(1))\sqrt{\frac{cn(n-cn)}{(cn-t)(n-cn+t)}}\left(\frac{cn}{cn-t}\right)^{cn-t}\left(\frac{n-cn}{n-cn+t}\right)^{n-cn}\left(\frac{cn}{n-cn+t}\right)^t.
	\]
	Therefore,
	\[
		\log \frac{\binom{n}{cn-t}}{\binom{n}{cn}}=O(1)-(cn-t)\log\left(1-\frac{t}{cn}\right)-(n-cn)\log\left(1+\frac{t}{n-cn}\right)-t\log\left(\frac{n-cn+t}{cn}\right).
	\]
	Because of the estimates $\log(1+1/x)=1/x+o(x^{-1})$ and $\log(1-1/x)=-1/x-o(x^{-1})$ for $x>1$, we see that
	\[
		\log \frac{\binom{n}{cn-t}}{\binom{n}{cn}}=o\left(\frac{n}{\log n}\right)+\frac{(cn-t)t}{cn}-t-t\log\left(\frac{n-cn+t}{cn}\right)=-\Omega\left(\frac{n}{\log n}\right),
	\]
	as desired.
\end{proof}

We are now prepared to establish the following asymptotic lower bound.

\begin{theorem} \label{thm:asymptotic result for alpha_{n,k} based on Shitov's work}
	For $n,k \in \nn$ such that $0.111n < k < 0.5n$, the following formula holds:
	\begin{equation*}
		\alpha_{n,k} \ge \binom{n}{k-1} \left(1 - \exp\left(-\Omega\left(\frac{n}{\log n}\right)\right)\right).
	\end{equation*}
\end{theorem}

\begin{proof}
%	\color{red}
%	In the proof of Theorem 5.6, there is a typo in the triple sum bound, the upper bound should be
%	$$n(k-d)\sum_{m=1}^{n-1}\sum_{a+b=k-d}\binom{m}{a}\binom{n-m}{b}$$
%	instead of
%	$$n^2(k-d)\sum_{a+b=t}\binom{m}{a}\binom{n-m}{b}.$$
%	\color{black}
	Let $d\in \rr$ be a constant that we will choose later. We will count the number of non-atoms. This is equal to the number of sets $A,B\subset \nn_0$ with $|A+B|=k$ and $\max A+\max B\le n$, such that $0\in A\cap B$. By \cite[Lemma~2.2]{yS14}, there are at most $n^{2d+4}2^{n/2}$ such pairs $(A,B)$ with $|A|+|B|\ge k-d$. The number of such pairs $(A,B)$ with $|A|+|B|<k-d$ is at most
	\[
		n\sum_{t=1}^{k-d}\sum_{m=1}^{n-1}\sum_{a+b=t}\binom{m}{a}\binom{n-m}{b} \le n(k-d)\sum_{m=1}^{n-1}\sum_{a+b=k-d}\binom{m}{a}\binom{n-m}{b}, % n^2(k-d)\sum_{a+b=t}\binom{m}{a}\binom{n-m}{b},
	\]
	where $m$ represents $\max A$. Note that $\sum_{a+b=t}\binom{m}{a}\binom{n-m}{b}$ is the coefficient of $x^t$ in $(1+x)^m(1+x)^{n-m}=(1+x)^n$, or $\binom{n}{t}$. Therefore, the number of non-atoms is at most
	\[
		n^{2d+4}2^{n/2}+n^2(k-d)\binom{n}{k-d}.
	\]
	In other words, for all $d\in \rr$,
	\[
		1-\frac{\alpha_{n,k}}{\binom{n}{k-1}}\le \frac{n^{2d+4}2^{n/2}}{\binom{n}{k-1}}+\frac{n^2(k-d)\binom{n}{k-d}}{\binom{n}{k-1}}.
	\]
	Let $d=\frac{\varepsilon n}{\log n}$ for a sufficiently small $\varepsilon$. Because $H(0.111)>\log(2^{1/2})$, it follows from Proposition~\ref{prop:asymptotics of binomial coefficients} that
	\[
		1-\frac{\alpha_{n,k}}{\binom{n}{k-1}}\le \exp(-\Omega(n))+\frac{n^2(k-d)\binom{n}{k-d}}{\binom{n}{k-1}}.
	\]
	Therefore by Proposition~\ref{prop:asymptotics of binomial coefficient ratio},
	\[
		1-\frac{\alpha_{n,k}}{\binom{n}{k-1}}\le \exp\left(-\Omega\left(\frac{n}{\log n}\right)\right),
	\]
	and the result follows.
\end{proof}

It is worth noting that Theorem~\ref{thm:asymptotic result for alpha_{n,k} based on Shitov's work} implies a weaker result regarding unimodality of $(\alpha_{n,k})_{k\ge 1}$.

\begin{cor} \label{cor:unimodality result for alpha_{n,k} based on Shitov's work}
	For each $n \in \nn$, the sequence $(\alpha_{n,k})_{k\ge 1}$ is increasing for $0.111n<k<0.5n$. In other words, the sequence is unimodal for at least $0.389n-O(1)$ values of $k$.
\end{cor}

\begin{proof}
	For $0.111n<k<0.5n$, we see that
	\[
		\alpha_{n,k+1} = \binom{n}{k}(1-o(1)) > \binom{n}{k}\frac{k}{n-k+1} = \binom{n}{k-1}\ge \alpha_{n,k},
	\]
	by Proposition~\ref{thm:asymptotic result for alpha_{n,k} based on Shitov's work}, as desired.
\end{proof}

\bigskip
%%%%%%%%%%%%%%%%%%
%%%%%%%%%%%%%%%%%%
\section{A Probabilistic Approach}
\label{sec:probabilistic approach}

In this section, we establish an asymptotic formula generalizing that one given in Theorem~\ref{thm:asymptotic result for alpha_{n,k} based on Shitov's work}. This result is highly motivated by \cite[Section~6.5]{BG24}. As a first application of the established asymptotic formula, we prove almost complete unimodality of the finite sequence $(\alpha_{n,k})_{1 \le k \le n+1}$ for any $n \in \nn$. This result generalizes Corollary~\ref{cor:unimodality result for alpha_{n,k} based on Shitov's work}. As a second application of the asymptotic formula, we later compute all the moments of the random variables $X_n \colon \mathcal{A}_n \to \nn_0$ defined by the assignments $A \mapsto |A|$, whose probability mass function is given by
\[
	\pp(X_{N,n}=k) = \frac{\alpha_{n,k}}{\alpha_n}.
\]

\medskip
%%%%%%%%%%%%%%%%%%%%%%%%%%%%%%%%
%%%%%%%%%%%%%%%%%%%%%%%%%%%%%%%%
\subsection{An Asymptotic Formula to Improve Unimodality} 

Before we establish an asymptotic formula for the sequence $(\alpha_{n,k})_{n,k \ge 1}$, we need to argue the following technical lemma, which is an analog of \cite[Lemma~6.5]{BG24} for our situation.

\begin{lemma} \label{lem:BG modification}
	Let $n\in \nn$ and $\epsilon\in \rr$ such that $0<\epsilon<1$. Let $b = (b_1, \dots, b_r)$ and $z = (z_1, \dots, z_r)$ be vectors in $\nn_0^r$ with $b_1 < \dots < b_r < n$ and $n - b_r = \Omega(n)$ while $z \in \{0,1\}^r$. Let $A=(a_1, \dots, a_n)$ be an element of $\{0,1\}^n$, chosen uniformly at random, such that $\sum_{i=1}^n a_i=\epsilon n$. Let $A_{b,z} \colon \{0,1\}^n \to \rr$ be the random variable denoting the number of indices $\ell$ such that $a_{\ell + b_i} = z_i$ (for every $i \in \ldb 1,r \rdb$). Then
	\[
		\pp\big( |A_{b,z} - \mathbb{E}(A_{b,z})| \ge \lambda \mathbb{E}(A_{b,z}) \big)
	\]
	is $O(n^{-1})$ for any $\lambda > 0$.
\end{lemma}

\begin{proof}
	By Chebyshev's inequality,
	\[
		\mathbb{P}(|A_{b,z}-\mathbb{E}(A_{b,z})| \ge \lambda\mathbb{E}(A_{b,z}))\le \frac{\mathrm{Var}(A_{b,z})}{\lambda^2 \mathbb{E}(A_{b,z})^2}.
	\]
	By linearity of the expectation, $\mathbb{E}(A_{b,z})=\Theta(n)$. As such, 
	\[
		\lambda^2 \mathbb{E}(A_{b,z})^2=\Theta(n^2).
	\]
	It hence suffices to show that $\mathrm{Var}(A_{b,z})=O(n)$. One can split
	\[
		A_{b,z}=\sum_{i=1}^{n-b_r}A^i_{b,z},
	\]
	where $A^i_{b,z}$ is the indicator random variable for the equality $(a_{i+b_1}, a_{i+b_2}, \dots)=(z_1, z_2, \dots)$. As a consequence, one obtains that
	\[
		\mathrm{Var}(A_{b,z})=\sum_{1\le i\le n-b_r}\mathrm{Var}(A^i_{b,z})+\sum_{1\le i_1<i_2\le n-b_r}\mathrm{Cov}(A^{i_1}_{b,z}, A^{i_2}_{b,z}).
	\]
	The covariance $\mathrm{Cov}(A^{i_1}_{b,z}, A^{i_2}_{b,z})$ is $0$ whenever the sequences corresponding to indices $i_1$ and $i_2$ do not intersect, because in this case $A^{i_1}_{b,z}$ and $A^{i_2}_{b,z}$ are independent. Since each sequence intersects $O(1)$ other sequences, there are $O(n)$ nonzero covariance terms. Because $A^i_{b,z}\in \{0,1\}$ for all $i \in \ldb 1, n-b_r\rdb$, the result follows.
\end{proof}

We are in a position to provide an asymptotic formula for $(\alpha_{n,k})_{n,k \ge 1}$.

\begin{theorem} \label{thm:asymptotic formula for alpha_{n,k} based on BG approach}
	Let $\epsilon$ be a positive constant such that $0<\epsilon<1$. Then there exists a function $\gamma_{\epsilon}:\nn \to \rr$ such that $\gamma_{\epsilon}(n)\to 0$ as $n\to\infty$, and
	\begin{equation*}
		\alpha_{n,\epsilon n} = \binom{n}{\epsilon n-1} (1 - \gamma_{\epsilon}(n))
	\end{equation*}
	for all $n\in \nn$ such that $\epsilon n$ is an integer. In fact, we may have $\gamma_{\epsilon}(n)=O(n^{-1})$.
\end{theorem}

\begin{proof}
%	\color{red}
%		This also isn't absolutely necessary, but in the proof of case 1 of Theorem 5.9, the double sum bound isn't very pretty as written. Using the $\substack$ command for the inner summation would make it look less clunky.
%	\color{black}
	Let $\gamma\in \mathbb{R}_{>0}$ and $r\in \mathbb{N}$ be constants that we will specify later. Let $n \in \nn$ such that $\epsilon n$ is an integer. We choose a subset $A\subseteq \ldb 0, n\rdb$ containing $0$ such that $|A| = \epsilon n$, and the goal is to show that $A$ is not an atom with probability $O(n^{-1})$. In other words, the probability that there exist sets $B,C$ with $A=B+C$ and $\min(|B|,|C|)\ge 2$ is $O(n^{-1})$. Define the following three sets:
	\begin{enumerate}
		\item Let $S_1$ be the set of such $A$ with a decomposition $A=B+C$ such that $|B|+|C|\le \gamma n$.
		\smallskip
		
		\item Let $S_2$ be the analogous set for $\min(|B|, |C|)\le r$.
		\smallskip
		
		\item Let $S_3$ be the set for $|B|+|C|>\gamma n$ and $\min(|B|,|C|)>r$.
	\end{enumerate}
	Note $S_1\cup S_2\cup S_3$ is the set of all $A$ that can be decomposed as $B+C$ with $\min(|B|,|C|)\ge 2$. Therefore, we wish to show that $\mathbb{P}(A\in S_1\cup S_2\cup S_3)=o(1)$. We split into three cases.
	\smallskip
	
	\textsc{Case 1}: To upper bound $\mathbb{P}(A\in S_1)$, we upper bound $|S_1|$ by counting the number of $B,C$ with $|B|+|C|\le \gamma n$. Observe that
	\[
		|S_1|\le \sum_{t\le \gamma n}\sum_{\substack{b+c=t \\ \min(b,c)\ge 2}} \binom{n}{b-1}\binom{n}{c-1} \le \sum_{t \le \gamma n}\sum_{\substack{b+c=t-2 \\ \min(b,c)\ge 1}}\binom{n}{b}\binom{n}{c}.
	\]
	Notice that $\sum_{b+c=t-2}\binom{n}{b}\binom{n}{c}$ is the coefficient of $x^{t-2}$ in $(1+x)^n\cdot (1+x)^n=(1+x)^{2n}$, which is $\binom{2n}{t-2}$. Hence
	\[
		|S_1|\le \sum_{t\le \gamma n}\binom{2n}{t-2}\le n\binom{2n}{\gamma n}
	\]
	with the assumption $\gamma<1$. Therefore
	\[
		\mathbb{P}(A\in S_1)\le \frac{n\binom{2n}{\gamma n}}{\binom{n}{\epsilon n-1}}.
	\]
	\smallskip
	
	\textsc{Case 2}: To compute $\mathbb{P}(A\in S_2)$, assume that $|B|\le r$. The case with $|C|\le r$ is disjoint and of equal probability because $|B+C|\le |B||C|$, and $\epsilon n\le r^2$ is impossible asymptotically as $r$ is a constant. Let $0<\omega<1$ be a constant that we will choose later. If $\max B\ge \omega n$ then $C\subseteq \ldb 0, (1-\omega)n\rdb$, so there are at most $n^r2^{(1-\omega)n}$ possible $A$. Now let us consider $\max B\le \omega n$. Write
	\[
		B=\{0,b_1,b_1+b_2,\dots,b_1+b_2+\dots+b_r\},
	\]
	where $b_1>0$ and $b_i\ge 0$ for all $i\ge 2$. For all $n_0\in A$, either $n_0+b_1\in A$ or some $n_0-b_i\in A$. Let $f:\mathbb{N}\to \{0,1\}$ be such that $f(x)=1$ if and only if $x\in A$. Therefore, for all $n_0\in \nn_0$,
	\[
		(f(n_0), f(n_0+b_1), f(n_0-b_1),\dots,f(n_0-b_r))\ne (1,0,0,\dots,0).
	\]
	By Lemma~\ref{lem:BG modification} with $\lambda=1$, there are $\binom{n}{\epsilon n-1}O(n^{-1})$ possible $A$ in this case. Hence
	$$\mathbb{P}(A\in S_2)\le O(n^{-1})+\frac{2n^r 2^{(1-\omega)n}}{\binom{n}{\epsilon n-1}}.$$
	\smallskip
	
	\textsc{Case 3}: To compute $\mathbb{P}(A\in S_3)$ assume without loss of generality that $|B|\ge \gamma n/2$ and $|C|\ge r$, by symmetry and the Pigeonhole Principle. Let $D\subseteq C$ such that $|D|=r$. Then for all $n_0\in B$, the inclusion $n_0+D \subseteq A$ holds. In other words, there are at least $\gamma n/2$ values of $n_0$ such that $f(n_0+d_i)=1$ for all $i$. On the other hand, the expected number of such $n_0$ is at most $(n-\max D)\epsilon^r\le n\epsilon^r$ by linearity of expectation and the fact that $\frac{\epsilon n-t}{n-t}<\epsilon$ for all $t\le r$. Hence by Lemma~\ref{lem:BG modification}, as long as $\gamma>2\epsilon^r$,
	\[
		\mathbb{P}(A\in S_3)=O(n^{-1}).
	\]
	Combining the cases, when $\gamma>2\epsilon^r$ we obtain that, by the union bound,
	\begin{align*}
		\mathbb{P}(A\in S_1\cup S_2\cup S_3)&\le \mathbb{P}(A\in S_1)+\mathbb{P}(A\in S_2)+\mathbb{P}(A\in S_3)  \\ &\le O(n^{-1})+\frac{n\binom{2n}{\gamma n}}{\binom{n}{\epsilon n-1}}+\frac{2n^r2^{(1-\omega)n}}{\binom{n}{\epsilon n-1}}.
	\end{align*}
	Therefore, the equality $\mathbb{P}(A\in S_1\cup S_2\cup S_3) = O(n^{-1})$ will follows after we verify that, for some positive $\varepsilon$, that
	\begin{enumerate}
		\item $\gamma>2\epsilon^r$,
		\smallskip
		
		\item $\varepsilon+\frac{1}{n}\log \binom{2n}{\gamma n}<\frac{1}{n}\log \binom{n}{\epsilon n}$, and
		\smallskip
		
		\item $\varepsilon+\frac{1}{n}\log(2n^r 2^{(1-\omega)n})<\frac{1}{n}\log \binom{n}{\epsilon n}$.
	\end{enumerate}
	By Proposition~\ref{prop:asymptotics of binomial coefficients}, (2) is implied by $2H(\gamma/2)<H(\epsilon)$, and (3) is implied by $n(1-\omega)\log 2<H(\epsilon)$. Clearly, we can choose $\gamma, \omega$ small enough such that (2) and (3) hold (because $H$ is increasing on $(0,1/2)$). Then, we may choose $r$ large enough so that (1) is satisfied, and we are done.
\end{proof}

As a consequence of Theorem~\ref{thm:asymptotic formula for alpha_{n,k} based on BG approach}, we obtain that, for each $n \in \nn$, the sequence $(\alpha_{n,k})_{k \ge 1}$ is ``almost unimodal".

\begin{cor}
	For each $n \in \nn$, there are $n(1-o(1))$ values of $k$ such that the unimodality inequality holds for $(\alpha_{n,k})_{n,k}$, that is, $\alpha_{n,k} < \alpha_{n,k+1}$ if $k < n/2$ and $\alpha_{n,k} > \alpha_{n,k+1}$ otherwise.
\end{cor}

\begin{proof}
	Let $\epsilon$ be a positive real constant. We prove that the appropriate unimodality inequality holds for $\epsilon n<k<(1-\epsilon)n$, which establishes the desired result as this statement holds true for all such $\epsilon$. There are two cases, which, in combination, prove the result.
	\smallskip

	\textsc{Case 1}: When $\epsilon n<k<n/2$, we see that $$\alpha_{n,k+1}=\binom{n}{k}(1-o(1))>\binom{n}{k}\frac{k}{n-k+1}=\binom{n}{k-1}\ge \alpha_{n,k}$$
	by Theorem~\ref{thm:asymptotic formula for alpha_{n,k} based on BG approach}.
	\smallskip

	\textsc{Case 2}: When $n/2\le k<(1-\epsilon)n$, we also have
	$$\alpha_{n,k+1}=\binom{n}{k}(1-o(1))<\binom{n}{k}\frac{k}{n-k+1}(1-o(1))=\binom{n}{k-1}(1-o(1))=\alpha_{n,k}$$
	by Theorem~\ref{thm:asymptotic formula for alpha_{n,k} based on BG approach}.
\end{proof}

\medskip
%%%%%%%%%%%%%%%%%%%%%%%%%%%%%%%%%%%%%%%%%%%%%%%%%%%
\subsection{The Random Variables Related to $(\alpha_{n,k})_{n,k \ge 1}$ and Their Moments}
%
%Call $S \in \Pfino(N)$ \emph{decomposable} if we can write $S = B+C$ in $\Pfino(N)$ for non-singletons $B$ and $C$. For each $n \in \nn_0$, we let $\text{Dec}(n)$ be the subset of $\Pfino(N)$ consisting of all decomposable $S \in \Pfino(N)$ with $\max S \le n$. In~\cite[Theorem~6.2]{BG24}, the authors proved that $\pp_{S \subseteq \ldb 0,n \rdb}(S \in \text{Dec(n)}) = \exp(-\Omega(n))$, and they use this result to argue that
%
%\color{red}
%Currently, equation (5.1) states that $\ell_N \in [1/2, 1] \cap \mathbb{Q}$. I think that equation (5.1) should be $\ell_N = 1$ instead, because it is for restricted power monoids, so we end up in the first case in Theorem 6.1 of the Bienvenu and Geroldinger paper (screenshotted below).
%
%It seems that equation (5.7) would need to be changed as well.
%\color{black}
%
%\begin{equation} \label{eq:Bienvenu-Geroldinger result}
%	\lim_{n \to \infty} \frac{|\aaa_n(N)|}{|\{A \in \Pfino(N) : \max A \le n\}|} \in \Big[\frac12, 1 \Big] \cap \qq,
%\end{equation}
%being $1$ precisely when the numerical monoid $N$ is $\nn_0$. When we replace $N$ by $\nn_0$ in~\eqref{eq:Bienvenu-Geroldinger result}, we obtain~\eqref{eq:Shitov's asymptotic for alpha(n)}. 

We proceed to consider the sequence of random variables $(X_n)_{n \ge 1}$, whose $n$-th term %corresponds to the partition $\{\aaa_{n,k} : k \in \ldb 1,n+1 \rdb\}$ of $\aaa_n$ and 
is defined as follows:
\[
	X_n \colon \aaa_n \to \nn_0, \quad \text{where} \quad X_n(A) = |A|.
\]
%In the especial case of the numerical monoid $\nn_0$, we write $X_n$ instead of $X_{\nn_0,n}$. 
For each $n \in \nn$, it is clear that $\sum_{k \in \nn_0} \pp(X_{N,n} = k) = 1$. Thus, the probability mass function of $X_{N,n}$ is given by
\[
	\pp(X_n = k) = \frac{|\aaa_{n,k}|}{|\aaa_n|} = \frac{\alpha_{n,k}}{\alpha_n}.
\]
We conclude this paper by proving that, for each $m \in \nn$, the sequence of moments $(\ee(X_n^m))_{n \ge 1}$ behaves asymptotically as that of a sequence $(\ee(Y_n^m))_{n \ge 1}$, where $Y_n$ is a binomially distributed random variable with parameters $n$ and $\frac12$.

\begin{theorem}
	For each $n \in \nn$, let $Y_n$ be a binomially distributed random variable with parameters $n$ and $\frac12$. Then, for each $r \in \nn$, the following equality holds:
	\begin{equation*}
		\lim_{n \to \infty} \frac{\mathbb{E}(X_n^r)}{\mathbb{E}(Y_n^r)} = 1.
	\end{equation*}
\end{theorem}

\begin{proof}
	For any $n,k \in \nn$, we write $\alpha_{n,k} = \binom{n}{k-1} - \beta_{n,k}$ for some $\beta_{n,k} \in \nn_0$. We have seen in~\eqref{eq:Shitov's asymptotic for alpha(n)} that $\alpha_n = 2^n(1 - o(1))$, which implies that 
	\begin{align*}
		\ee(X_n^r) \! = \! \frac{\sum_{k=1}^{n+1} k^r \alpha_{n,k}}{\alpha_n} \! = \! \frac{\sum_{k=1}^{n+1} k^r \alpha_{n,r}}{2^n(1 - o(1))}
		\! = \! \frac{\sum_{k=1}^{n+1} k^r \binom{n}{k-1}}{2^n(1 - o(1))} - \frac{\sum_{k=1}^{n+1} k^r \beta_{n,k}}{2^n(1 - o(1))}
		\! = \! \frac{\ee(Y_n^r)}{2^{n-1}(1 - o(1))} - \frac{\sum_{k=1}^{n+1} k^r \beta_{n,k}}{2^n(1 - o(1))},
	\end{align*}
	where the last equality follows from the identity $\ee(Y_n^r) = \frac12 \sum_{k=0}^n k^r \binom{n}{k}$. As a consequence,
	\[
		\lim_{n \to \infty} \frac{\ee(X_n^r)}{\ee(Y_n^r)} = \lim_{n \to \infty} \frac{2^{1-n}}{(1 - o(1))} \bigg( 1 - \frac{\sum_{k=1}^{n+1} k^r \beta_{n,k}}{\sum_{k=1}^{n+1} k^r \binom{n}{k-1}} \bigg) = 1-s(n), %1 - \frac{\sum_{k=0}^n k^r \beta_{n,k}}{\sum_{k=0}^n k^r \binom{n}{k}} = 1-s  \quad \text{where} \quad s := \frac{\sum_{k=0}^n k^r \beta_{n,k}}{\sum_{k=0}^n k^r \binom{n}{k}},
	\]
	where $s(n) := \sum_{k=1}^{n+1} k^r \beta_{n,k}/ \sum_{k=1}^{n+1} k^r \binom{n}{k-1}$. Now fix $\epsilon = \frac1{10}$, and then set
	\[
		s_1(n) := \frac{\sum_{k=1}^{\epsilon n} k^r \beta_{n,k}}{\sum_{k=1}^{n+1} k^r \binom{n}{k-1}},   \quad s_2(n) := \frac{\sum_{k=\epsilon n}^{(1-\epsilon)n} k^r \beta_{n,k}}{\sum_{k=1}^{n+1} k^r \binom{n}{k-1}},  \quad  \text{and} \quad s_3(n) := \frac{\sum_{k=(1-\epsilon)n}^{n+1} k^r \beta_{n,k}}{\sum_{k=1}^{n+1} k^r \binom{n}{k-1}}.
	\]
	Observe that the equality $s(n) = s_1(n) + s_2(n) + s_3(n)$ holds, so our proof is complete once we argue that $s_i(n) \to 0$ as $n \to \infty$ for every $i \in \ldb 1,3 \rdb$.
	
	We first argue that $s_2(n) \to 0$ as $n \to \infty$. For any $n \in \nn$ and $\epsilon\in \rr$ with $0<\epsilon<1$, the equality $\beta_{n,\epsilon n} = \binom{n}{\epsilon n-1}\big(1 - \alpha_{n,\epsilon n}/\binom{n}{\epsilon n-1}\big)$, along with Theorem~\ref{thm:asymptotic formula for alpha_{n,k} based on BG approach}, ensures the existence of a function $\gamma_{\epsilon} \colon \nn \to \rr$ with $\gamma_{\epsilon}(n) \to 0$ as $n \to \infty$ such that $\beta_{n,\epsilon n} = \binom{n}{\epsilon n-1} \gamma_{\epsilon}(n)$. Now consider the function $\gamma \colon \nn \to \rr$ defined by
	\[
		\gamma(n) = \max\big\{\gamma_{\varepsilon}(n) : \varepsilon \in [\epsilon, 1-\epsilon] \big\}.
	\]
	Observe that $\gamma(n) \to 0$ as $n \to \infty$. As a result,
	\[
		\lim_{n \to \infty} s_2(n) = \lim_{n \to \infty} \frac{\sum_{k=\epsilon n}^{(1-\epsilon)n} k^r \beta_{n,k}}{\sum_{k=1}^{n+1} k^r \binom{n}{k-1}} = \lim_{n \to \infty} \frac{\sum_{k=\epsilon n}^{(1-\epsilon)n} k^r \binom{n}{k-1} \gamma_{k/n}(n)}{\sum_{k=1}^{n+1} k^r \binom{n}{k-1}} \le \lim_{n \to \infty} \gamma(n) \frac{\sum_{k=\epsilon n}^{(1-\epsilon)n} k^r \binom{n}{k-1}}{\sum_{k=1}^{n+1} k^r \binom{n}{k-1}} = 0.
	\]
	Now, we will argue that $s_1(n), s_3(n) \to 0$ as $n \to \infty$. For all $n,k \in \nn$, the equality $\binom{n}{k-1} - \beta_{n,k} = \alpha_{n,k} > 0$ guarantees that the following inequalities hold:
	\[
		s'_1(n) := \frac{\sum_{k=1}^{\epsilon n} k^r \binom{n}{k-1}}{\sum_{k=1}^{n+1} k^r \binom{n}{k-1}} \ge\frac{\sum_{k=1}^{\epsilon n} k^r \beta_{n,k}}{\sum_{k=1}^{n+1} k^r \binom{n}{k-1}} = s_1(n)
	\]
	and
	\[
		\quad s'_3(n) := \frac{\sum_{k=(1-\epsilon)n}^{n+1} k^r \binom{n}{k-1}}{\sum_{k=1}^{n+1} k^r \binom{n}{k-1}} \ge \frac{\sum_{k=(1-\epsilon)n}^{n+1} k^r \beta_{n,k}}{\sum_{k=1}^{n+1} k^r \binom{n}{k-1}} = s_3(n).
	\]
	Because $s'_1(n) \le s'_3(n)$ for every $n \in \nn$, it suffices to show that $s'_3(n) \to 0$ as $n \to \infty$. Towards this, we first observe that
	\[
		s'_3(n) \le n^r \frac{\sum_{k=1}^{\epsilon n} \binom{n}{k}}{2^n} \le n^r 2 \exp\Big(-2\Big(\frac12 - \epsilon \Big)^2 n\Big) = 2n^r \exp(-0.32n),
	\]
	where the second inequality follows after applying Hoeffding's inequality to the sum of uniform random variables. As a consequence,
	\[
		\lim_{n \to \infty} s'_3(n) = \lim_{n \to \infty} n^r \exp(-0.32n) = 0,
	\]
	and so $s_1(n), s_3(n) \to 0$ as $n \to \infty$. Hence $\lim_{n \to \infty} s(n) = 0$, and the proof is complete.
\end{proof}

\bigskip
%%%%%%%%%%%%%%%
%%%%%%%%%%%%%%%
\section*{Acknowledgments}

This collaboration took place as part of CrowdMath, a year-long online math research program hosted by the MIT Math Department and the AoPS. All the authors would like to express their gratitude to the directors and organizers of CrowdMath for making this research experience possible. The authors are also grateful to Alfred Geroldinger and Salvatore Tringali for kindly providing some of the questions motivating the research project of CrowdMath 2023 and, therefore, this collaboration. In addition, the authors would like to thank Marly Gotti for carefully proof/reading earlier versions of this paper. During the period of this collaboration, the second author was kindly supported by the NSF award DMS-2213323.

\bigskip
%%%%%%%%%%%%%%%
%%%%%%%%%%%%%%%
\section*{Conflict of Interest Statement}

On behalf of all authors, the corresponding author states that there is no conflict of interest related to this paper.

\bigskip
%%%%%%%%%%%%%%%
%%%%%%%%%%%%%%%

\end{document}